\documentclass[11pt]{article}
\usepackage{amsmath,amsthm, amssymb, amsfonts,accents,nccmath}
\usepackage{enumitem}
\usepackage{array}
\usepackage{lipsum}
\usepackage{mathtools}
\usepackage[toc,page]{appendix}
\usepackage[english]{babel}
\usepackage{dsfont}
\usepackage{booktabs}
\usepackage[linesnumbered,commentsnumbered,ruled,vlined]{algorithm2e}
\usepackage{tikz}
\usetikzlibrary{decorations.pathreplacing,calligraphy}
\usepackage[utf8]{inputenc}
\usepackage[english]{babel}

\usepackage{caption}
\usepackage{subcaption}

\newcommand\restr[2]{{% we make the whole thing an ordinary symbol
  \left.\kern-\nulldelimiterspace % automatically resize the bar with \right
  #1 % the function
  \vphantom{\big|} % pretend it's a little taller at normal size
  \right|_{#2} % this is the delimiter
  }}

\usepackage{hyperref}
\hypersetup{
    colorlinks=true,
    linkcolor=blue,
    filecolor=darkmagenta,      
    urlcolor=cyan,
    citecolor=red
} 
\makeatletter
\def\BState{\State\hskip-\ALG@thistlm}
\makeatother
\numberwithin{equation}{section}

\makeatletter
\tikzset{
	dot diameter/.store in=\dot@diameter,
	dot diameter=2pt,
	dot spacing/.store in=\dot@spacing,
	dot spacing=9pt,
	dots/.style={
		line width=\dot@diameter,
		line cap=round,
		dash pattern=on 0pt off \dot@spacing
	}
}
\makeatother

\newtheorem{theorem}{Theorem}[section]

\newtheorem{lem}[theorem]{Lemma}

\newtheorem{rem}[theorem]{Remark}

\title{On Sombor Index of Unicyclic graphs with a fixed number of pendant vertices}
\author{
	Joyentanuj Das\thanks{Department of Applied Mathematics, National Sun Yat-sen University, Kaohsiung City - 804, Taiwan (R.O.C.). \newline Email: joyentanuj@gmail.com, joyentanuj@math.nsysu.edu.tw}
	\and
	Yogesh Prajapaty\thanks{School of Mathematics, IISER Thiruvananthapuram, Maruthamala P.O., Vithura, Kerala- 695 551, India. \newline Email: prajapaty0916@iisertvm.ac.in}
}
\date{}

\begin{document}

\maketitle

\begin{abstract}
The Sombor index is a topological index in graph theory defined by Gutman in 2021. In this article we find the maximum Sombor index of unicyclic graphs with a fixed number of pendant vertices. We also provide the unique graph among the chosen class where the maximum Sombor index is attained.
\end{abstract}



\section{Introduction}
Let $G=(V(G),E(G))$  be a finite, simple, connected graph with $V(G)$ as the set of vertices and $E(G)$ as the set of edges in $G$. We simply write $G=(V,E)$ if there is no scope of confusion. We write $u\sim v$ to indicate that the vertices $u,v \in V$ are adjacent in $G$. The degree
of the vertex $v$, denoted by $d_G(v)$ (or simply $d(v)$), equals the number of vertices in $V$ that are adjacent to $v$.  A graph $H$ is said to be a subgraph of $G$ if $V(H) \subset V(G)$ and $E(H) \subset E(G)$. For any subset $S \subset V (G)$, a subgraph $H$ of $G$ is said to be an induced subgraph with vertex set $S$, if $H$ is a maximal subgraph of $G$ with vertex set $V(H)=S$.

For any two vertices $u,v \in V(G)$, we will use $d(u,v)$ to denote the distance between the vertices $u$ and $v$, i.e. the length of the shortest path connecting $u$ and $v$. For a vertex $u \in V(G)$, we use $N(u)$ to denote the set of neighbors of $u$. \textit{i.e.} the set of vertices that are adjacent to $u$. A vertex $u \in V(G)$ is said to be a pendant vertex if $d(u) = 1$. We will use $C_l,P_l$ to denote a cycle and a path of length $l$, respectively.

In~\cite{Gut1}, Gutman defined a new topological index under the name Sombor index. For a graph G, its Sombor index is defined as $$SO(G) = \sum_{u \sim v} \sqrt{d(u)^2 + d(v)^2}.$$ Note that in $SO(G)$, the summation is over all edges of $G$. This new index attracted many researchers within a short period of time and some of them were able to show applications of the Sombor index, for example see~\cite{Reti}. For articles related to Sombor index the readers can refer to \cite{Ali},\cite{Cruz}-\cite{Zhou1}. 

In graph theory studying extremal graphs and indices for a class of graphs with a given parameter is a very interesting problem. For example one may refer to \cite{Das1,Sen,Sun,Zhou,Zhou1}. In \cite{Zhou} the authors have studied the extremal Sombor index of trees and unicyclic graphs with given matching number. In \cite{Zhou1}, the authors have studied the Sombor index of trees and unicyclic graphs with a given maximum degree. In~\cite{Sen}, the authors found the maximum Sombor index of unicyclic graphs with a fixed girth. In this article we find the maximum Sombor index of unicyclic graphs with a fixed number of pendant vertices and also provide the class of graphs where the maximum Sombor index is attained.

This article is organized as follows: In Section~\ref{sec:notations}, we state the preliminary results, prove some basic inequalities and results related to trees that will be used in the later section. In Section~\ref{sec:main}, we build up the necessary tools required to prove the main result and lastly in Theorem~\ref{thm:main}, we state and prove the main result of the article.

\section{Notation and Preliminaries}\label{sec:notations}
The next two lemmas directly follows from verifying $f'(x) >0$ and $g'(x) < 0$ and these will be used again and again in the future proofs.
\begin{lem}\label{lem:f}
	For $x \ge 1$, $f(x) = \sqrt{(x+a)^2+b^2}-\sqrt{x^2+b^2}$, where $a,b$ are positive integers, then $f(x)$ is a monotonically increasing function on $x$.
\end{lem}

\begin{lem}\label{lem:g}
	For $x \ge 1$, $g(x) = \sqrt{a^2+ x^2}-\sqrt{b^2+x^2}$, where $a,b$ are positive integers and $a>b$, then $g(x)$ is a monotonically decreasing function on $x$.
\end{lem}

The following two lemmas gives us some basic inequalities that would be required in the proofs of some of the results later in this article.
\begin{lem}\label{lem:ine1}
	Let $c,d \ge 2$ be integers then the following inequality holds true $$\sqrt{c^2 + 2^2} + \sqrt{d^2 + 2^2} \ge \sqrt{c^2 + d^2} + \sqrt{2^2+2^2}.$$
\end{lem}

\begin{proof}
	Squaring both sides of the inequality we have 
	$$c^2 + 2^2 + d^2 + 2^2 + 2\sqrt{c^2 + 2^2} \sqrt{d^2 + 2^2} \ge c^2 + d^2 +  2^2 + 2^2 + 2\sqrt{c^2 + d^2} \sqrt{2^2+2^2}.$$ Canceling the common terms we have $\sqrt{c^2 + 2^2} \sqrt{d^2 + 2^2} \ge \sqrt{c^2 + d^2} \sqrt{2^2+2^2}$. Again squaring both sides we have $$(c^2+2^2)(d^2 + 2^2) \ge (c^2 + d^2)(2^2+2^2),$$ which is true if and only if $c^2d^2 + 16 \ge 4(c^2 + d^2)$. Note that $$c^2d^2 - 4(c^2 + d^2) + 16 = (c^2 -2^2)(d^2 - 2^2) \ge 0,$$ since $c,d \ge 2$ and hence the inequality holds true.
\end{proof}

\begin{lem}\label{lem:ine2}
	Let $m,n \ge 1$ be integers then the following inequalities hold true:
	\begin{itemize}
		\item[(a)] $\sqrt{(m+n)^2 + 2^2} \ge \sqrt{(m+1)^2 + (n+1)^2}.$
		\item[(b)] $\sqrt{(m+n+2)^2 + 2^2} \ge \sqrt{(m+1)^2 + (n+3)^2}.$
		\item[(c)] $\sqrt{(m+n+2)^2 + 2^2} \ge \sqrt{(m+2)^2 + (n+2)^2}.$
	\end{itemize}
\end{lem}

\begin{proof}
	\begin{itemize}
		\item[(a)] Squaring both sides of the inequality we have $$(m+n)^2 + 2^2 \ge (m+1)^2 + (n+1)^2.$$ Expanding the terms on both sides we have $$m^2 + 2mn + n^2+2^2 \ge m^2 + 2m + 1+n^2+2n+1.$$ Simplifying by canceling similar terms on both sides we have $mn+2 \ge m+n$, which is true and hence the result holds.
		\item[(b)] Squaring both sides of the inequality we have $$(m+n+2)^2 + 2^2 \ge (m+1)^2 + (n+3)^2 .$$ Expanding the terms on both sides we have $$m^2 + n^2 + 2^2 + 4(m+n) + 2mn + 2^2\ge  m^2 + 2m + 1 + n^2 + 6n + 3^2.$$ Simplifying by canceling similar terms on both sides we have $m(n+1)\ge n+1$, which is true and hence the result holds.
		\item[(c)] We skip the proof since it is similar to that of (a) and (b).
	\end{itemize}
\end{proof}

Let $T$ be a tree and $u, v \in V(T)$ be vertices such that $d(u) \ge 3$ and $v$ is the nearest vertex to $u$ with $d(v) \ge 3$. Let the path containing $u,v$ be $\cdots \sim u_1 \sim u \sim v \sim v_1 \sim \cdots$. Let $N(u)\setminus \{v\} = \{u_1,u_2,\cdots,u_m\}$, $N(v)\setminus \{u\} = \{v_1,v_2,\cdots,v_n\}$ and $m \ge n$.

\begin{lem}\label{lem:tree1}
	Let $T$ be a tree as given above and $T^*$ be the graph obtained from $T$ by deleting the edges $v \sim v_i$ for all $i \ge 2$ and adding the edges $u \sim v_i$ for all $i \ge 2$. Then we have $SO(T^*) > SO(T)$.
\end{lem}

\begin{proof}
	Without loss of generality we can assume that $d(v_1) \ge d(v_i)$ for all $i \ge 2$. We prove the result by showing that $SO(T^*) - SO(T) > 0$. Note that 
	\begin{align*}
		SO(T^*) - SO(T) &> \sqrt{(m+n)^2 + 2^2} - \sqrt{(m+1)^2 + (n+1)^2} \\
						&\ \ +\sqrt{(m+n)^2 + d(v_2)^2} - \sqrt{(n+1)^2 + d(v_2)^2}\\
						&\ \ +\sqrt{2^2 + d(v_1)^2} - \sqrt{(n+1)^2 + d(v_1)^2}
	\end{align*}
Using Lemma~\ref{lem:f} and substituting $a = m-1$ and $b = d(v_2)$ we have 
$$\sqrt{(m+n)^2 + d(v_2)^2} - \sqrt{(n+1)^2 + d(v_2)^2} > \sqrt{(m+1)^2 + d(v_2)^2} - \sqrt{2^2 + d(v_2)^2},$$ since $f(n+1) > f(2)$. Next, using the condition that $m \ge n$ we have $$\sqrt{(m+1)^2 + d(v_2)^2} - \sqrt{2^2 + d(v_2)^2} \ge \sqrt{(n+1)^2 + d(v_2)^2} - \sqrt{2^2 + d(v_2)^2}.$$ Finally using Lemma~\ref{lem:g} and substituting $a = n+1$ and $b = 2$ we have $$\sqrt{(n+1)^2 + d(v_2)^2} - \sqrt{2^2 + d(v_2)^2} \ge \sqrt{(n+1)^2 + d(v_1)^2} - \sqrt{2^2 + d(v_1)^2},$$ since $g(d(v_2)) \ge g(d(v_1))$. Thus combining the above inequalities we have $$\sqrt{(m+n)^2 + d(v_2)^2} - \sqrt{(n+1)^2 + d(v_2)^2} > \sqrt{(n+1)^2 + d(v_1)^2} - \sqrt{2^2 + d(v_1)^2}.$$ It follows from Lemma~\ref{lem:ine2} that $\sqrt{(m+n)^2 + 2^2} \ge  \sqrt{(m+1)^2 + (n+1)^2}$. Hence we have $SO(T^*) - SO(T) > 0$ and the result follows.
\end{proof}

Let $T$ be a tree and $u, v \in V(T)$ be vertices such that $d(u) \ge 3$ and $v$ is the nearest vertex to $u$ with $d(v) \ge 3$. Let the path containing $u,v$ be $\cdots \sim u_1 \sim u \sim u_2 \sim \cdots \sim v_2 \sim v \sim v_1 \sim \cdots$ and $N(u)\setminus \{v\} = \{u_1,u_2,\cdots,u_m\}$, $N(v)\setminus \{u\} = \{v_1,v_2,\cdots,v_n\}$ and $m \ge n$.

\begin{lem}\label{lem:tree2}
	Let $T$ be a tree as given above and $T^*$ be the graph obtained from $T$ by deleting the edges $v \sim v_i$ for all $i \ge 3$ and adding the edges $u \sim v_i$ for all $i \ge 3$. Then we have $SO(T^*) > SO(T)$.
\end{lem}

\begin{proof}
	Observe that, since $u_2,v_2$ lies on the path between the vertices $u$ and $v$ we have $d(u_2) = d(v_2) = 2$. Without loss of generality we can assume that $d(v_1) \ge d(v_i)$ for all $i \ge 3$. We prove the result by showing that $SO(T^*) - SO(T) > 0$. Note that
	\begin{align*}
		SO(T^*) - SO(T) &> \sqrt{(m+n-2)^2 + 2^2} - \sqrt{m^2 + 2^2} \\
		&\ \ +\sqrt{2^2 + 2^2} - \sqrt{n^2 + 2^2}\\
		&\ \ +\sqrt{(m+n-2)^2 + d(v_3)^2} - \sqrt{n^2 + d(v_3)^2}\\
		&\ \ +\sqrt{2^2 + d(v_1)^2} - \sqrt{n^2 + d(v_1)^2}
	\end{align*}
Using Lemma~\ref{lem:f} and substituting $a = n-2$ and $b = 2$ we have  $$\sqrt{(m+n-2)^2 + 2^2} - \sqrt{m^2 + 2^2} > \sqrt{n^2 + 2^2} - \sqrt{2^2 + 2^2},$$ since $f(m) > f(2)$. Again using Lemma~\ref{lem:f} and substituting $a = m-2$ and $b = d(v_3)$ we have $$\sqrt{(m+n-2)^2 + d(v_3)^2} - \sqrt{n^2 + d(v_3)^2} > \sqrt{m^2 + d(v_3)^2} - \sqrt{2^2 + d(v_3)^2},$$ since $f(n) > f(2)$. Next, using the condition that $m \ge n$ we have $$\sqrt{m^2 + d(v_3)^2} - \sqrt{2^2 + d(v_3)^2} \ge \sqrt{n^2 + d(v_3)^2} - \sqrt{2^2 + d(v_3)^2}.$$ Finally, using Lemma~\ref{lem:g} and substituting $a = n$ and $b = 2$ we have $$\sqrt{n^2 + d(v_3)^2} - \sqrt{2^2 + d(v_3)^2} \ge \sqrt{n^2 + d(v_1)^2} - \sqrt{2^2 + d(v_1)^2},$$ since $g(d(v_3)) \ge g(d(v_1))$. Thus combining the above inequalities we have $SO(T^*) - SO(T) > 0$ and the result follows.
\end{proof}

\begin{rem}
	If in Lemma~\ref{lem:tree2} we have $u_2 = v_2$, i.e. $u_2$ and $v_2$ are identical vertices then also the result is true and a similar argument gives us the result.
\end{rem}

\section{Main Results}\label{sec:main}
Let $\mathcal{U}(N,k)$ be the class of unicyclic graphs on $N$ vertices with $k$ pendant vertices. Note that since there are $k$ pendant vertices we have $N \ge k+3$. Let $U_3(N,k)$ be the class of unicyclic graphs on $N$ vertices with $k$ pendant vertices where the cycle is of length $3$, say $C_3$ and there are $k$ paths attached to a single vertex of $C_3$.

In this section we find the maximum Sombor index of the class of graphs $\mathcal{U}(N,k)$. Next we prove two lemmas which shows how Sombor index changes if we remove an edge from a path and add it to a different path where both paths are attached to a single vertex of the cycle.

Let $G \in \mathcal{U}(N,k)$ and $P_l,P_m$ be two paths of length $l \ge 3$ and $m \ge 2$ attached to a single vertex of $G$ on the cycle. Let $v_l$ be the pendant vertex of $P_l$ with $v_l \sim v_{l-1} \sim v_{l-2}\in P_l$ and $u_m$ be the pendant vertex of $P_m$ with $u_m \sim u_{m-1}$.

\begin{lem}\label{lem:U3}
	Let $G$ defined be as above and $G^*$ be the graph obtained from $G$ by deleting the edge $v_l \sim v_{l-1}$ and adding the edge $u_m \sim v_l$, then $SO(G) = SO(G^*)$.
\end{lem}

\begin{proof}
	We prove the result by showing that $SO(G) - SO(G^*) = 0$. Since $P_l$ and $P_m$ has length at least $3$ and $2$, respectively we have that $d(v_{l-2}) = d(v_{l-1}) = 2$ and $d(v_{m-1}) = 2$. Note that 
	\begin{align*}
		SO(G) - SO(G^*) &= \sqrt{2^2+1} + \sqrt{2^2 + 2^2} + \sqrt{2^2+1}\\
		&\ \ -\sqrt{2^2+2^2} - \sqrt{2^2 + 1} - \sqrt{2^2+1} = 0
	\end{align*} and the result follows.
\end{proof}

Let $G \in \mathcal{U}(N,k)$ and $P_l,P_m$ be two paths of length $l = 2$ and $m \ge 2$ attached to a single vertex $u \in V(G)$ on the cycle. Let $v_l$ be the pendant vertex of $P_l$ with $v_l \sim v_{l-1}\in P_l$ and $u_m$ be the pendant vertex of $P_m$ with $u_m \sim u_{m-1}$.

\begin{lem}\label{lem:U4}
	Let $G$ be defined as above and $G^*$ be the graph obtained from $G$ by deleting the edge $v_l \sim v_{l-1}$ and adding the edge $u_m \sim v_l$, then $SO(G^*) > SO(G)$.
\end{lem}

\begin{proof}
	We prove the result by showing that $SO(G) - SO(G^*) > 0$. Since $l = 2$ and $m \ge 2$ we have $d(v_{l-1}) = 2$ and $d(v_{m-1}) = 2$. Also observe that since $u$ belongs to the cycle and has at least two paths attached to it $d(u) \ge 4$. Thus we have
		\begin{align*}
		SO(G^*) - SO(G) &= \sqrt{d(u)^2+1} + \sqrt{2^2 + 2^2} -\sqrt{d(u)^2+2^2} - \sqrt{2^2 + 1}\\
		&= \sqrt{d(u)^2+1} + \sqrt{8} -\sqrt{d(u)^2+4} - \sqrt{5} > 0.
	\end{align*} The last inequality follows from the observation that $8(d(u)^2+1) > 5(d(u)^2+4)$ since $d(u) \ge 4$ and the result follows.
\end{proof}

Let $\mathcal{G} \in U_3(N,k) \subset \mathcal{U}(N,k)$ be the graph that contains a cycle of length $3$ along with $k-1$ pendant edges and a path of length $N-k-2$ attached to a single vertex of the cycle. For reference one can refer to Figure~\ref{fig:unicyclic}.

\begin{figure}[ht]
	\centering
	\begin{tikzpicture}
		\draw[fill=black] (3,-2) circle (2pt);
		\draw[fill=black] (1,-2) circle (2pt);
		\draw[fill=black] (0,0) circle (2pt);
		\draw[fill=black] (1,2) circle (2pt);
		\draw[fill=red] (2,0) circle (3pt);
		\draw[fill=black] (3,0) circle (2pt);
		\draw[fill=black] (4,0) circle (2pt);
		\draw[fill=black] (7,0) circle (2pt);
		\draw[fill=black] (8,0) circle (2pt);
		
		\draw[thick] (0,0) -- (1,0) -- (2,0) -- (4,0);
		\draw[thick] (1,-2) -- (2,0) -- (3,-2); 
		\draw[thick] (7,0) -- (8,0);
		\draw[thick] (2,0) -- (1,2) -- (0,0);
		
		\node at (2,-2.3) {$k-1$};
		\node at (5.5,-0.3) {$N-k-2$};

		\draw [dots]  (4,0) -- (7,0);
		\draw [dots]  (3,-2) -- (1,-2);			
	\end{tikzpicture}
\caption{$k-1$ pendant edges and a path of length $N-k-2$ attached to a single vertex of a $C_3$.} \label{fig:unicyclic}
\end{figure}
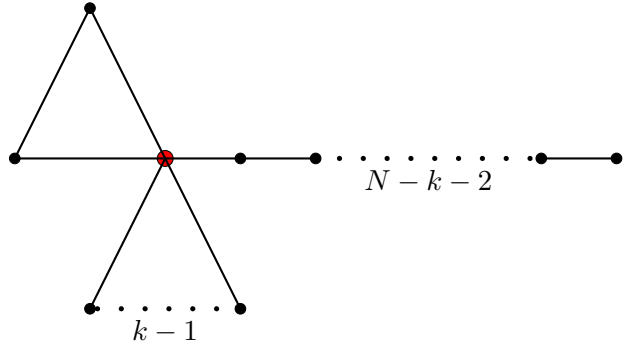

From the Figure~\ref{fig:unicyclic}, one can observe that the degree of the vertex(colored red) on $C_3$ where the pendant edges and the path is attached has degree $k+2$. In the next remark we state the value of $SO(\mathcal{G})$.
\begin{rem}
	Let $\mathcal{G} \in U_3(N,k) \subset \mathcal{U}(N,k)$ be defined as above, then
	$$SO(\mathcal{G}) =  \begin{cases}
		(N-k-3) \sqrt{2^2 + 2^2} + (k-1) \sqrt{(k+2)^2 + 1} + 3 \sqrt{(k+2)^2 + 2^2} + \sqrt{2^2 + 1}& \text{ if } N \ge k+ 4\\
		k \sqrt{(k+2)^2 + 1} + 2 \sqrt{(k+2)^2 + 2^2} + \sqrt{2^2 + 2^2}& \text{ if } N = k+ 3.
	\end{cases}$$
	
\end{rem}

Let $G \in \mathcal{U}(N,k)$ be such that it contains a cycle $C_l$. Let $x \sim u \sim y$ be vertices on the cycle and $T$ be a tree that is attached to $u$ with $u \sim v$ where $v \in T$ and $d(v) \ge 3$. Let $N(u) = \{x,y,v\} \cup \{u_1,u_2,\cdots,u_n\}$ and $N(v) = \{u\} \cup \{v_1,v_2,\cdots,v_m\}$. Let $T^*$ be the tree component containing $v$ that is obtained by deleting the edge $u \sim v$. Observe that by Lemmas~\ref{lem:tree1} and \ref{lem:tree2} we can conclude that $v$ is the only vertex in $T^*$ that has degree $\ge 3$.

\begin{lem}\label{lem:uni1}
	Let $G \in \mathcal{U}(N,k)$ be the graph as defined above.  Then we have the following:
	\begin{itemize}
		\item[(a)] If $d(u) \ge d(v)$, let $G^*$ be the graph obtained from $G$ by deleting the edges $v \sim v_i$ and adding the edges $u \sim v_i$ for all $i \ge 2$, then we have $SO(G^*) > SO(G)$.
		\item[(b)] If  $d(v) > d(u)$, let $G^*$ be the graph obtained from $G$ by deleting the edges $u \sim u_i$, $u \sim x$ and adding the edges $v \sim u_i$, $v \sim x$ for all $i \ge 1$. Then $SO(G^*) > SO(G)$.
	\end{itemize}
\end{lem}

\begin{proof}
	\begin{itemize}
		\item[(a)] We skip the proof since the argument follows similarly as in Lemma~\ref{lem:tree1}.
		\item[(b)] Without loss of generality we assume that $d(y) \ge d(x)$. Since $d(v) > d(u)$ we have $m > n+2$. We complete the proof by showing $SO(G^*) - SO(G) > 0$. Note that
		\begin{align*}
			SO(G^*) - SO(G)  &> \sqrt{2^2 + d(y)^2} - \sqrt{(n+3)^2 + d(y)^2}\\
			&\ \ +\sqrt{(m+n+2)^2 + d(x)^2} - \sqrt{(n+3)^2 + d(x)^2}\\
			&\ \ +\sqrt{(m+n+2)^2 + 2^2} - \sqrt{(n+3)^2 + (m+1)^2}.
		\end{align*}
		Using Lemma~\ref{lem:f} and substituting $a = m-1$ and $b = d(x)$ we have $$\sqrt{(m+n+2)^2 + d(x)^2} - \sqrt{(n+3)^2 + d(x)^2} > \sqrt{(m+1)^2 + d(x)^2} - \sqrt{2^2 + d(x)^2},$$ since $f(n+3) > f(2)$. Since $m > n+2$ we have $$\sqrt{(m+1)^2 + d(x)^2} - \sqrt{2^2 + d(x)^2} > \sqrt{(n+3)^2 + d(x)^2} - \sqrt{2^2 + d(x)^2}.$$ Next using Lemma~\ref{lem:g}, substituting $a = n+3$ and $b = 2$ and the fact $d(y) \ge d(x)$ we have $$\sqrt{(n+3)^2 + d(x)^2} - \sqrt{2^2 + d(x)^2} \ge \sqrt{(n+3)^2 + d(y)^2} - \sqrt{2^2 + d(y)^2},$$ since $g(d(x)) \ge g(d(y))$. Thus combining the inequalities we have $$\sqrt{(m+n+2)^2 + d(x)^2} - \sqrt{(n+3)^2 + d(x)^2} > \sqrt{(n+3)^2 + d(y)^2} - \sqrt{2^2 + d(y)^2}.$$ It follows from Lemma~\ref{lem:ine2} that $\sqrt{(m+n+2)^2 + 2^2} \ge \sqrt{(n+3)^2 + (m+1)^2}$. Hence the result follows.
	\end{itemize}
\end{proof}

Let $G \in \mathcal{U}(N,k)$ be such that it contains a cycle $C_l$. Let $x \sim u \sim y$ be vertices on the cycle and $T$ be a tree that is attached to $u$ with $u \sim w \sim v$ where $v,w \in T$ and $d(v) \ge 3$. Let $N(u) = \{x,y,w\} \cup \{u_1,u_2,\cdots,u_n\}$ and $N(v) = \{w\} \cup \{v_1,v_2,\cdots,v_m\}$. Let $T^*$ be the tree component containing $w$ and $v$ that is obtained by deleting the edge $u \sim w$. Observe that by Lemmas~\ref{lem:tree1} and \ref{lem:tree2} we can conclude that $v$ is the only vertex in $T^*$ that has degree $\ge 3$ and this implies $d(w) = 2$.

\begin{lem}\label{lem:uni2}
	Let $G \in \mathcal{U}(N,k)$ be the graph as defined above. Then we have the following:
	\begin{itemize}
		\item[(a)] If $d(u) \ge d(v)$, then consider $G^*$ to be the graph obtained from $G$ by deleting the edges $v \sim v_i$ and adding the edges $u \sim v_i$ for all $i \ge 2$, then we have $SO(G^*) > SO(G)$.
		\item[(b)] If $d(v) > d(u)$, then consider $G^*$ to be the graph obtained from $G$ by deleting the edges $u \sim u_i$, $u \sim x$ and adding the edges $v \sim u_i$, $v \sim x$ for all $i \ge 1$. Then $SO(G^*) > SO(G)$.
	\end{itemize}
\end{lem}

\begin{proof}
\begin{itemize}
	\item[(a)] We skip the proof since the argument follows similarly as in Lemma~\ref{lem:tree2}.
	\item[(b)] Without loss of generality we assume that $d(y) \ge d(x)$. Since $d(v) > d(u)$ we have $m > n+2$. We complete the proof by showing $SO(G^*) - SO(G) > 0$. Note that
	\begin{align*}
		SO(G^*) - SO(G)  &> \sqrt{2^2 + d(y)^2} - \sqrt{(n+3)^2 + d(y)^2}\\
		&\ \ +\sqrt{(m+n+2)^2 + d(x)^2} - \sqrt{(n+3)^2 + d(x)^2}\\
		&\ \ +\sqrt{(m+n+2)^2 + 2^2} - \sqrt{(m+1)^2 + 2^2}\\
		&\ \ +\sqrt{2^2+2^2} - \sqrt{2^2 + (n+3)^2}.
	\end{align*}
	Using Lemma~\ref{lem:f} and substituting $a = n+1$ and $b = 2$ we have $$\sqrt{(m+n+2)^2 + 2^2} - \sqrt{(m+1)^2 + 2^2} > \sqrt{(n+3)^2 + 2^2} - \sqrt{2^2 + 2^2},$$ since $f(m+1) > f(2)$. Next in the proof of Lemma~\ref{lem:uni1} we have already observed that $$\sqrt{(m+n+2)^2 + d(x)^2} - \sqrt{(n+3)^2 + d(x)^2} > \sqrt{(n+3)^2 + d(y)^2} - \sqrt{2^2 + d(y)^2}.$$ Hence the result follows.
\end{itemize}
\end{proof}

\begin{rem}\label{rem:uni3}
	In Lemmas~\ref{lem:uni1} and \ref{lem:uni2}, we have consider $G \in \mathcal{U}(N,k)$  such that a tree $T$ is attached to a vertex $u$ on the cycle and the vertex $v \in T$ that has degree $\ge 3$ is such that $d(u,v) = 1$ and $d(u,v) = 2$, respectively. Same argument follows if $d(u,v) \ge 3$ since all the vertices on the path from $u$ to $v$ has degree equal to $2$, by Lemmas~\ref{lem:tree1} and \ref{lem:tree2}.
\end{rem}

Let $G \in \mathcal{U}(N,k)$ be such that it contains a cycle $C_l$ of length at least $4$. Let $x \sim u \sim v \sim y$ be vertices on $C_l$ in $G$ such that there are paths attached to $u$ and $v$. Let $N(u) = \{x,v\} \cup \{u_1,u_2,\cdots,u_n\}$ and $N(v) = \{u,y\} \cup \{v_1,v_2,\cdots,v_m\}$.

\begin{lem}\label{lem:cycle1}
	Let $G$ be the graph defined as above and $G^*$ be the graph obtained from $G$ by deleting the edge $u \sim v$, merging the two vertices and adding a new vertex to a pendant vertex of a tree attached to $u$ or $v$, say $z$. Then we have $SO(G^*) > SO(G)$.
\end{lem}

\begin{proof}
	Given that $z$ is a pendant vertex on a path we assume that the vertex adjacent to $z$ be $t$. We prove the result by showing that $SO(G^*) - SO(G) > 0$. 
	
\underline{\textbf{Case 1:}} $d(t) = 2$. We have
	\begin{align*}
		SO(G^*) - SO(G) &> \sqrt{(m+n+2)^2 + d(u_1)^2} - \sqrt{(n+2)^2 + d(u_1)^2}\\
		&\ \ + \sqrt{(m+n+2)^2 + d(v_1)^2} - \sqrt{(m+2)^2 + d(v_1)^2}\\
		&\ \ +\sqrt{2^2+1}- \sqrt{(m+2)^2 + (n+2)^2} + \sqrt{2^2 + d(t)^2} - \sqrt{d(t)^2 + 1}\\
		&= \sqrt{(m+n+2)^2 + d(u_1)^2} - \sqrt{(n+2)^2 + d(u_1)^2}\\
		&\ \ + \sqrt{(m+n+2)^2 + d(v_1)^2} - \sqrt{(m+2)^2 + d(v_1)^2}\\
		&\ \ - \sqrt{(m+2)^2 + (n+2)^2} + \sqrt{2^2 + 2^2}
	\end{align*}
Using Lemma~\ref{lem:f} and substituting $a = m$ and $b = d(u_1)$ we have $$\sqrt{(m+n+2)^2 + d(u_1)^2} - \sqrt{(n+2)^2 + d(u_1)^2} > \sqrt{(m+2)^2 + d(u_1)^2} - \sqrt{2^2 + d(u_1)^2},$$ since $f(n+2) > f(2)$. Again using Lemma~\ref{lem:f} and substituting $a = n$ and $b = d(v_1)$ we have $$\sqrt{(m+n+2)^2 + d(v_1)^2} - \sqrt{(m+2)^2 + d(v_1)^2} > \sqrt{(n+2)^2 + d(v_1)^2} - \sqrt{2^2 + d(v_1)^2},$$ since $f(m+2) > f(2)$. Adding these two inequalities, we have
\begin{equation}\label{eqn:1}
	\begin{split}
			\sqrt{(m+n+2)^2 + d(u_1)^2} - \sqrt{(n+2)^2 + d(u_1)^2} + \sqrt{(m+n+2)^2 + d(v_1)^2} - \sqrt{(m+2)^2 + d(v_1)^2}\\
		> \sqrt{(m+2)^2 + d(u_1)^2} - \sqrt{2^2 + d(u_1)^2} + \sqrt{(n+2)^2 + d(v_1)^2} - \sqrt{2^2 + d(v_1)^2}.
	\end{split}
\end{equation}

Observe that $u_1,v_1$ are vertices on the path and hence $d(u_1),d(v_1) \le 2$. Thus using Lemma~\ref{lem:g} and substituting $a = m+2$ and $b = 2$ we have $$\sqrt{(m+2)^2 + d(u_1)^2} - \sqrt{2^2 + d(u_1)^2} \ge \sqrt{(m+2)^2 + 2^2} - \sqrt{2^2 + 2^2},$$ since $g(d(u_1)) \ge g(2)$. Using similar argument and using the fact $d(v_1) \le 2$ we have $$\sqrt{(n+2)^2 + d(v_1)^2} - \sqrt{2^2 + d(v_1)^2} \ge \sqrt{(n+2)^2 + 2^2} - \sqrt{2^2 + 2^2}.$$ Adding the above two inequalities, we have
\begin{equation}\label{eqn:2}
	\begin{split}
			\sqrt{(m+2)^2 + d(u_1)^2} - \sqrt{2^2 + d(u_1)^2} + \sqrt{(n+2)^2 + d(v_1)^2} - \sqrt{2^2 + d(v_1)^2}\\
		\ge \sqrt{(m+2)^2 + 2^2} - \sqrt{2^2 + 2^2} + \sqrt{(n+2)^2 + 2^2} - \sqrt{2^2 + 2^2}
	\end{split}
	\end{equation}

Combining the inequalities \eqref{eqn:2} and \eqref{eqn:2} one can see that to prove the fact that $SO(G) - SO(G^*) > 0$ it is enough to show that 
\begin{align*}
	\sqrt{(m+2)^2 + 2^2} - \sqrt{2^2 + 2^2} + \sqrt{(n+2)^2 + 2^2} - \sqrt{2^2 + 2^2} &\ge \sqrt{(m+2)^2 + (n+2)^2} -  \sqrt{2^2+2^2}
\end{align*} i.e. we have to prove that $$\sqrt{(m+2)^2 + 2^2} + \sqrt{(n+2)^2 + 2^2} \ge \sqrt{(m+2)^2 + (n+2)^2} + \sqrt{2^2+2^2}.$$ But this is true and the fact follows from Lemma~\ref{lem:ine1} by substituting $c = m+2$ and $d = n+2$. 

\underline{\textbf{Case 2:}} There exists no vertex $t$ among the paths such that $d(t) = 2$. This implies that all the paths are actually pendant edges and $z \sim u$ or $z \sim v$. In this case we have
\begin{align*}
	SO(G^*) - SO(G) &> \sqrt{(m+n+2)^2 + 2^2} - \sqrt{(n+2)^2 + 1}\\
	&\ \ + \sqrt{(m+n+2)^2 + 1} - \sqrt{(m+2)^2 +1}\\
	&\ \ +\sqrt{2^2+1}- \sqrt{(m+2)^2 + (n+2)^2}.
\end{align*}
Using Lemma~\ref{lem:f} and substituting $a = n$ and $b = 1$ we have $$\sqrt{(m+n+2)^2 + 1} - \sqrt{(m+2)^2 + 1} > \sqrt{(n+2)^2 + 1} - \sqrt{2^2 + 1},$$ since $f(m+2) > f(2)$. Using this fact we have 
\begin{align*}
	SO(G^*) - SO(G)&> \sqrt{(m+n+2)^2 + 2^2} - \sqrt{(n+2)^2 + 1}\\
	&\ \ + \sqrt{(n+2)^2 + 1} - \sqrt{2^2 + 1} +\sqrt{2^2+1}- \sqrt{(m+2)^2 + (n+2)^2}\\
	&\ \  = \sqrt{(m+n+2)^2 + 2^2} -\sqrt{(m+2)^2 + (n+2)^2} \ge 0.
\end{align*}
The last inequality follows from Lemma~\ref{lem:ine2}. Thus combining both cases we have $SO(G^*) - SO(G) > 0$ and hence the result follows.
\end{proof}

Let $G \in \mathcal{U}(N,k)$ be such that it contains a cycle $C_l$ of length at least $4$. Let $ u \sim w \sim v $ be vertices on $C_l$ in $G$ such that there are paths attached to $u$ and $v$ and $d(w) = 2$. Let $N(u) = \{x,w\} \cup \{u_1,u_2,\cdots,u_n\}$ and $N(v) = \{w,y\} \cup \{v_1,v_2,\cdots,v_m\}$.

\begin{lem}\label{lem:cycle2}
	Let $G$ be the graph defined as above and $G^*$ be the graph obtained from $G$ by deleting the edges $u \sim w$ and $w \sim v$, merging the vertices $u,v$ and adding a path of length $2$ to a pendant vertex, say $z$ of a path attached to $u$ or $v$. Then we have $SO(G^*) > SO(G)$.
\end{lem}

\begin{proof}
Given that $z$ is a pendant vertex on a path we assume that the vertex adjacent to $z$ be $t$.	We prove the result by showing that $SO(G^*) - SO(G) > 0$.

\underline{\textbf{Case 1:}} $d(t) = 2$. We have
	\begin{align*}
		SO(G^*) - SO(G) &> \sqrt{(m+n+2)^2 + d(u_1)^2} - \sqrt{(n+2)^2 + d(u_1)^2}\\
		&\ \ + \sqrt{(m+n+2)^2 + d(v_1)^2} - \sqrt{(m+2)^2 + d(v_1)^2}\\
		&\ \ +\sqrt{2^2+1} + \sqrt{2^2 + 2^2}- \sqrt{(m+2)^2 + 2^2} - \sqrt{(n+2)^2+2^2}\\
		&\ \ + \sqrt{2^2 + d(t)^2} - \sqrt{d(t)^2 + 1}\\
		&= \sqrt{(m+n+2)^2 + d(u_1)^2} - \sqrt{(n+2)^2 + d(u_1)^2}\\
		&\ \ + \sqrt{(m+n+2)^2 + d(v_1)^2} - \sqrt{(m+2)^2 + d(v_1)^2}\\
		&\ \ + 2\sqrt{2^2 + 2^2}- \sqrt{(m+2)^2 + 2^2} - \sqrt{(n+2)^2+2^2}.
	\end{align*}
Using the fact $d(u_1),d(v_1) \le 2$ along with similar argument as in Lemma~\ref{lem:cycle1} we have 
\begin{align*}
	&\sqrt{(m+n+2)^2 + d(u_1)^2} - \sqrt{(n+2)^2 + d(u_1)^2} +  \sqrt{(m+n+2)^2 + d(v_1)^2} - \sqrt{(m+2)^2 + d(v_1)^2}\\
	&> \sqrt{(m+2)^2 + d(u_1)^2} - \sqrt{2^2 + d(u_1)^2} + \sqrt{(n+2)^2 + d(v_1)^2} - \sqrt{2^2 + d(v_1)^2}\\
	&\ge \sqrt{(m+2)^2 + 2^2} - \sqrt{2^2 + 2^2} + \sqrt{(n+2)^2 + 2^2} - \sqrt{2^2 + 2^2}
\end{align*} and hence we have the required inequality.

\underline{\textbf{Case 2:}} There exists no vertex $t$ among the paths such that $d(t) = 2$. This implies that all the paths are actually pendant edges and $z \sim u$ or $z \sim v$. In this case we have
\begin{align*}
	SO(G^*) - SO(G) &> \sqrt{(m+n+2)^2 + 2^2} - \sqrt{(n+2)^2 + 1}\\
	&\ \ + \sqrt{(m+n+2)^2 + 1} - \sqrt{(m+2)^2 +1}\\
	&\ \ +\sqrt{2^2+2^2}+\sqrt{2^2+1}- \sqrt{(m+2)^2 + 2^2} - \sqrt{(n+2)^2 + 2^2}.
\end{align*}
Using Lemma~\ref{lem:f} and substituting $a = n$ and $b = 1$ we have $$\sqrt{(m+n+2)^2 + 1} - \sqrt{(m+2)^2 + 1} > \sqrt{(n+2)^2 + 1} - \sqrt{2^2 + 1},$$ since $f(m+2) > f(2)$. Using this fact and simplifying we have 
\begin{align*}
	SO(G^*) - SO(G) &> \sqrt{(m+n+2)^2 + 2^2} - \sqrt{(n+2)^2 + 1} + \sqrt{(n+2)^2 + 1} - \sqrt{2^2 + 1}\\
	&\ \ +\sqrt{2^2+2^2}+\sqrt{2^2+1}- \sqrt{(m+2)^2 + 2^2} - \sqrt{(n+2)^2 + 2^2}\\
	&= \sqrt{(m+n+2)^2 + 2^2} +\sqrt{2^2+2^2} - \sqrt{(m+2)^2 + 2^2} - \sqrt{(n+2)^2 + 2^2}.
\end{align*}
Using Lemma~\ref{lem:f} and substituting $a = n$ and $b = 2$ we have $$\sqrt{(m+n+2)^2 + 2^2} - \sqrt{(m+2)^2 + 2^2} > \sqrt{(n+2)^2 + 2^2} - \sqrt{2^2 + 2^2},$$ since $f(m+2) > f(2)$. Hence we have the required inequality. 

Thus combining both cases we have $SO(G^*) - SO(G) > 0$ and hence the result follows.
\end{proof}

\begin{rem}\label{rem:cycle3}
	In Lemmas~\ref{lem:cycle1} and \ref{lem:cycle2} we have considered the cases when $d(u,v) = 1$ and $d(u,v) = 2$. If $d(u,v) \ge 3$, then without loss of generality we can assume that all the vertices on the $uv$ path are of degree equal to $2$. Then consider the graph $G^*$ by deleting the $uv$ path, merging the vertices $u$ and $v$ and attaching a path of same length at a pendant vertex of a path attached to either $u$ or $v$. The proof of this case follows from the proof of Lemmas~\ref{lem:cycle1} and \ref{lem:cycle2}.
\end{rem}

Let $G \in \mathcal{U}(N,k)$ be such that it contains the cycle $C_3 = u \sim v \sim w \sim u$ along with some (or no) paths attached to $u,v$ and $w$. Without loss of generality we assume that $d(u) \ge d(v) \ge d(w)$. Let $N(u) = \{v,w\} \cup \{u_1,u_2,\cdots,u_n\}$ and $N(v) = \{w,u\} \cup \{v_1,v_2,\cdots,v_m\}$. 

\begin{lem}\label{lem:C3}
Let G be the graph defined as above and $G^*$ be the graph obtained from $G$ by deleting the edges $v \sim v_i$ and adding the edges $u \sim v_i$ for all $i \ge 1$. Then $SO(G^*) > SO(G)$.
\end{lem}

\begin{proof}
We prove the result by showing that $SO(G^*) - SO(G) > 0$. Observe that 
\begin{align*}
	SO(G^*) - SO(G) &> \sqrt{(m+n+2)^2 + 2^2} -\sqrt{(m+2)^2 + (n+2)^2}\\
				    &\ \ +\sqrt{(m+n+2)^2 + d(w)^2} - \sqrt{(n+2)^2 + d(w)^2}\\
				    &\ \ +\sqrt{2^2 + d(w)^2} - \sqrt{(m+2)^2 + d(w)^2}
\end{align*}
Using Lemma~\ref{lem:f} and substituting $a = m$ and $b = d(w)$ we have $$\sqrt{(m+n+2)^2 + d(w)^2} - \sqrt{(n+2)^2 + d(w)^2} > \sqrt{(n+2)^2 + d(w)^2} - \sqrt{2^2 + d(w)^2},$$ since $f(n+2) > f(2)$. Finally using Lemma~\ref{lem:ine2} we have $$\sqrt{(m+n+2)^2 + 2^2} -\sqrt{(m+2)^2 + (n+2)^2} \ge 0$$ and hence the result follows.
\end{proof} 

Now we prove the main result of the article.

\begin{theorem}\label{thm:main}
	Let $G \in \mathcal{U}(N,k)$. Then $SO(\mathcal{G}) \ge SO(G)$ and equality holds if and only if $G \cong \mathcal{G}$.
\end{theorem}

\begin{proof}
	Let $G \in \mathcal{U}(N,k)$ be a graph with maximal Sombor index. Using Lemmas~\ref{lem:uni1}, \ref{lem:uni2} and Remark~\ref{rem:uni3} we can conclude that $G$ is an unicyclic graph with $k$ paths attached to the vertices of the cycle in $G$. Next, using Lemma~\ref{lem:cycle1}, \ref{lem:cycle2} and Remark~\ref{rem:cycle3} we can conclude that the cycle in $G$ is of length $3$ with $k$ paths attaches to its vertices. Finally using Lemma~\ref{lem:C3} since $G$ is a maximal graph all $k$ paths must be attached to a single vertex of the cycle of length $3$. Thus $G \in U_3(N,k)$. 
	\begin{itemize}
	     \item[(a)] If $N = k+3$, then $G$ has only pendant edges attached to a single vertex of $C_3$ and $G \cong \mathcal{G}$..
	     \item[(b)] If $N = k+4$, then $G$ has $k-1$ pendant edges and a path of length $2$, all attached to a single vertex of $C_3$, \textit{i.e.} $G \cong \mathcal{G}$.
	     \item[(c)] If $N > k+4$ and $G$ has $k-1$ pendant edges and a path of length $N-k-2$ attached to a single vertex of $C_3$, then $G \cong \mathcal{G}$. Otherwise, $G$ has at least two paths of length $\ge 2$ then using Lemmas~\ref{lem:U3} and \ref{lem:U4} we have if $G$ is the maximal graph then $G \cong \mathcal{G}$.
	\end{itemize}
	Thus combining the above three cases we have the required result.
\end{proof}

\small{

}

\end{document}